\documentclass[a4paper,10pt]{article}
\usepackage[T1]{fontenc}
\usepackage[latin1]{inputenc}
\usepackage[english]{babel}
\usepackage{amsthm}
\usepackage{enumerate}
\usepackage{amssymb, amsmath, amsthm}
\usepackage{url}

\theoremstyle{definition}

\theoremstyle{remark}

\theoremstyle{theorem}

\newtheorem{theorem}{Theorem}
\newtheorem*{principle1}{First Principle}
\newtheorem*{principle2}{Second Principle}
\newtheorem*{principle3}{Third Principle}
\newtheorem*{principle4}{Fourth Principle}

\theoremstyle{definition}

\newtheorem{remark}[theorem]{Remark}

\def\eps{\varepsilon}
\def\RR{\mathbb{R}}
\def\NN{\mathbb{N}}
\def\ovr{\overline}
\def\cM{\mathcal{M}}

\begin{document}

\title{Littlewood's fourth principle}
\markright{Littlewood's fourth principle}
\author{Rolando Magnanini\thanks{Dipartimento di Matematica e Informatica ``U.~Dini'',
Universit\` a di Firenze, viale Morgagni 67/A, 50134 Firenze, Italy 
({\tt magnanin@math.unifi.it}).} and Giorgio Poggesi\thanks{Dipartimento di Matematica e Informatica ``U.~Dini'',
Universit\` a di Firenze, viale Morgagni 67/A, 50134 Firenze, Italy
({\tt giorgio.poggesi@stud.unifi.it}).}}

\maketitle

\begin{abstract}
In Real Analysis, Littlewood's three principles are known as heuristics that help teach the essentials of measure theory and reveal the analogies between the concepts of topological space and continuos function on one side and those of measurable space and measurable function on the other one. They are based on important and rigorous statements, such as Lusin's and Egoroff-Severini's theorems, and have ingenious and elegant proofs. We shall comment on those theorems and show how their proofs 
can possibly be made simpler by introducing a \textit{fourth principle}. These alternative
proofs make even more manifest those analogies and show that Egoroff-Severini's theorem can be considered the natural generalization of the classical Dini's monotone convergence theorem. 
\end{abstract}

\section{Introduction.}

John Edenson Littlewood (9 June 1885 - 6 September 1977) was a British mathematician.
In 1944, he wrote an influential textbook, \textit{Lectures on the Theory of Functions} (\cite{Li}), in which he proposed three principles as guides for working in real analysis;  these are heuristics to help teach the essentials of measure theory, as Littlewood himself wrote in \cite{Li}:
\begin{quotation}
The extent of knowledge [of real analysis] required is nothing like so great as is sometimes supposed. There are three principles, roughly expressible in the following terms:
every (measurable) set is nearly a finite sum of intervals;
every function (of class $L^{\lambda}$) is nearly continuous;
every convergent sequence is nearly uniformly convergent.
Most of the results of the present section are fairly intuitive applications of these ideas, and the student armed with them should be equal to most occasions when real variable theory is called for. If one of the principles would be the obvious means to settle a problem if it were ``quite'' true, it is natural to ask if the ``nearly'' is near enough, and for a problem that is actually soluble it generally is.
\end{quotation}

To benefit our further discussion, we shall express Littlewood's principles 
and their rigorous statements in forms
that are slightly different from those originally stated.

The first principle descends directly from the very definition of (Lebesgue) measurability of a set.

\begin{principle1}
Every measurable set is nearly closed.
\end{principle1}

The second principle relates the measurability of a function to the more familiar property of continuity.

\begin{principle2}
Every measurable function is nearly continuous.
\end{principle2}

The third principle connects the pointwise convergence of a sequence of functions to
the standard concept of uniform convergence.

\begin{principle3}
Every sequence of measurable functions that converges pointwise almost everywhere is nearly 
uniformly convergent.
\end{principle3}

These principles are based on important theorems that give a rigorous meaning to the term ``nearly''.
We shall recall these in the next section along with their ingenious proofs that give a taste of the 
standard arguments used in Real Analysis. 
\par
In Section 3, we will discuss a {\it fourth principle} that associates the concept of finiteness
of a function to that of its boundedness.

\begin{principle4}
Every measurable function that is finite almost everywhere is nearly 
bounded.
\end{principle4}

In the mathematical literature (see \cite{DiB}, \cite{CD}, \cite{Li}, \cite{Ma}, \cite{Ro}, \cite{Ru}, \cite{Ta}), the proof of the second principle is based on the third; it can be easily seen that the fourth principle can be derived from the second.
\par
However, we shall see that the fourth principle can also be proved independently;
this fact makes possible a proof of the second principle \textit{without} appealing for the third,
that itself can be derived from the second, by a totally new proof based on \textit{Dini's monotone convergence theorem}.
\par
 As in \cite{Li}, to make our discussion as simple as possible, we shall consider the Lebesgue measure $m$ for the real line $\mathbb {R}$; then in Section 4  
we shall hint at how the four principles and their rigorous counterparts can be extended to more general contexts.

\section{The three principles}
We recall the definitions of  {\it inner} and {\it outer measure} of
a set $E\subseteq\RR$: they are respectively\footnote{The number $|K|$ is the infimum of the
total lengths of all the finite unions of open intervals that contain $K$. 
Accordingly, $|A|$ is the supremum of the
total lengths of all the finite unions of closed intervals contained in $A$.}
\begin{eqnarray*}
&&m_i(E)=\sup\{ |K|: K \mbox{ is compact and } K\subseteq E\},\\
&&m_e(E)=\inf\{|A| : A\mbox{ is open and } A\supseteq E\},
\end{eqnarray*}
It always holds that $m_i(E)\le m_e(E)$. The set $E$ is (Lebesgue) measurable if and only if 
$m_i(E)=m_e(E)$;
when this is the case, the {\it measure} of $E$ is $m(E)=m_i(E)=m_e(E)$; thus $m(E)\in[0,\infty]$ 
and it can be proved that $m$ is a measure on the $\sigma$-algebra of Lebesgue measurable subsets of $\RR$, as specified in Section \ref{sec:extensions}.
\par
By the properties of the supremum, it is easily seen that, for any pair of subsets $E$ and $F$ of $\RR$, $m_e(E\cup F)\le m_e(E)+m_e(F)$ and $m_e(E)\le m_e(F)$ if $E\subseteq F$.
\vskip.1cm
The first principle is a condition for the measurability of subsets of $\RR$. 

\begin{theorem}[First Principle]
\label{th:first}
Let $E\subset\RR$ be a set of finite outer measure. 
\par
Then, $E$ is measurable if and only if for every $\eps > 0$ there exist two sets $K$ and $F$, with 
$K$ closed (compact), $K \cup F = E$ and $m_e(F) < \eps$.
\end{theorem}

This is what is meant for \textit{nearly closed}.

\begin{proof}
If $E$ is measurable, for any $\eps>0$ we can find a compact set $K\subseteq E$ and an open
set $A\supseteq E$ such that
$$
m(K)>m(E)-\eps/2 \ \mbox{ and } \ m(A)<m(E)+\eps/2.
$$
The set $A\setminus K$ is open and contains $E\setminus K$. Thus, by setting $F=E\setminus K$, we have $E=K\cup F$ and
$$
m_e(F)\le m(A)-m(K)<\eps.
$$
\par
Viceversa, for every $\eps>0$ we have: 
$$
m_e(E)=m_e(K\cup F)\le m_e(K)+m_e(F)<m(K)+\eps\le m_i(E)+\eps.
$$
Since $\eps$ is arbitrary, then $m_e(E)\le m_i(E)$.
\end{proof}

The second and third principles concern measurable functions from (measurable) subsets of $\RR$
to the {\it extended real line} $\ovr{\RR}=\RR\cup\{+\infty\}\cup\{-\infty\}$,
that is functions are allowed to have values $+\infty$ and $-\infty$. 
\par
Let $f:E\to\ovr{\RR}$ be a function defined on a measurable subset $E$ of $\RR$. 
We say that $f$ is \textit{measurable} if the \textit{level sets} defined by
$$
L(f,t)=\{ x\in E: f(x)>t\}
$$
are measurable subsets of $\RR$ for every $t\in\RR$. It is easy to verify that if we replace $L(f,t)$ with $L^*(f,t)=\{ x\in E: f(x) \geq t\}$ we have an equivalent definition.  
\par
Since the countable union of measurable sets is measurable, it is not hard to show that
the pointwise infimum and supremum of a sequence of measurable functions $f_n:E\to\ovr{\RR}$ 
are measurable functions as well as the function defined for any $x\in E$ by
$$
\limsup_{n\to\infty} f_n(x)=\inf_{k\ge 1}\sup_{n\ge k} f_n(x).
$$
\par
Since the countable union of sets of measure zero has measure zero and the difference between $E$ and any set of measure zero is measurable, the same definitions and conclusions hold even
if the functions $f$ and $f_n$ are defined \textit{almost everywhere}\footnote{Denoted for short by \textit{a.e.}}, that is
if the subsets of $E$ in which they are not defined has measure zero.\footnote{In the same spirit,
we say that a function or a sequence of functions satisfies a given property \textit{a.e. \!\!in $E$}, if that property holds with the exception of a subset of measure zero.}
\par
As already mentioned, the third principle is needed to prove the second and is
known as \textit{Egoroff's theorem} or \textit{Egoroff-Severini's theorem}.\footnote{Dmitri Egoroff, a Russian physicist and geometer and Carlo Severini, an Italian mathematician, published independent proofs of this theorem respectively in 1910 and 1911
(see \cite{Eg} and \cite{Se}); Severini's assumptions are more restrictive. Severini's result is not very well-known, since it is hidden in a paper on orthogonal polynomials, published in Italian.}

\begin{theorem}[Third Principle; Egoroff-Severini]
\label{th:Egoroff}
Let $E\subset \RR$ be a measurable set with finite measure and
let $f:E\to\ovr{\RR}$ be measurable and finite a.e. in $E$.

The sequence of measurable functions $f_n:E\to\ovr{\RR}$  converges a.e. to $f$ in $E$ for $n\to\infty$ if and only if, for every $\eps> 0$, there exists a closed set $K \subseteq E$ such that $m(E \setminus K) < \eps$ and $f_n$ converges uniformly to $f$ on $K$.
\end{theorem}

This is what we mean for \textit{nearly uniformly convergent}.

\begin{proof}
If $f_n \to f$ a.e. in $E$ as $n\to\infty$, the subset of $E$ in which $f_n \to f$ pointwise
has the same measure as $E$; hence, without loss of generality, we can assume that $f_n(x)$
converges to $f(x)$ for any $x\in E$.

Consider the functions defined by
\begin{equation}
\label{defgn}
g_n(x)=\sup_{k\ge n} |f_k(x)-f(x)|, \ \ x\in E
\end{equation}
and the sets
\begin{equation}
\label{defEnm}
E_{n,m} =\left\{x \in E : g_n(x) <\frac1{m} \right\}
\ \mbox{ for } \ n, m\in\NN.
\end{equation}
Observe that, if $x\in E$, then $g_n(x)\to 0$ as $n\to\infty$ and hence for any $m\in\NN$
$$
E=\bigcup_{n=1}^\infty E_{n,m}.
$$
As $E_{n,m}$ is increasing with $n$, the monotone convergence theorem implies that $m(E_{n,m})$ converges to $m(E)$ for $n\to\infty$  and for any $m \in\NN$.
Thus, for every $\eps>0$ and $m \in\NN$, there exists an index $\nu=\nu(\eps, m)$ such that
$
m(E \setminus E_{\nu,m}) < \eps/ {2^{m+1}}.
$

\par
The measure of the set $F = \bigcup\limits_{m = 1}^{\infty} (E \setminus E_{\nu,m})$ is arbitrary small, in fact
$$
m(F) \leq \sum_{m=1}^{\infty} m(E \setminus E_{\nu,m}) < \eps/ 2.
$$ 
Also, since $E\setminus F$ is measurable, by Thorem \ref{th:first}  there exists a compact set $K \subseteq E \setminus F$ such that $m(E \setminus F) - m(K) < \eps/ 2$, and hence 
$$
m(E \setminus K) = m(E\setminus F)+m(F) - m(K) < \eps.
$$
\par
Since $K \subseteq E \setminus F = \bigcap\limits_{m=1}^{\infty} E_{\nu(\eps, m), m}$ we have that 
$$
|f_n(x) - f(x)| < \frac{1}{m} \ \mbox{ for any } \ x\in K \ \mbox{ and } \ n \ge \nu(\eps, m),
$$ 
by the definitions of $E_{\nu,m}$ and $g_{n}$; this means that $f_n$ converges uniformly to $f$ on $K$ as $n\to\infty$.
\vskip.1cm
Viceversa, if for every $\eps> 0$ there is a closed set $K \subseteq E$ with
$m(E \setminus K) < \eps$ and $f_n \rightarrow f$ uniformly on $K$, then by
choosing $\eps=1/m$ we can say that there is a closed set $K_m \subseteq E$ such that $f_n \to f$ uniformly on $K_m$ and $m(E \setminus K_m) < 1 / m$. 
\par
Therefore, $f_n(x) \rightarrow f(x)$ for any $x$ in the set $F = \bigcup\limits_{m=1}^\infty K_m$ and 
$$
m(E \setminus F) = m\Bigl(\bigcap_{m=1}^\infty (E \setminus K_m)\Bigr) \leq m(E \setminus K_m) <\frac1{m} \ \mbox{ for any } \ m\in\NN,
$$ 
which implies that $m(E \setminus F) = 0$. Thus, $f_n\to f$ a.e. in $E$ as $n\to\infty$. 

\end{proof}

The second principle corresponds to \textit{Lusin's theorem} (see \cite{Lu}),\footnote{N. N. Lusin or Luzin was a student of Egoroff. For biographical notes on Egoroff and Lusin see \cite{GK}.} that we state
here in a form similar to Theorems \ref{th:first} and \ref{th:Egoroff}.

\begin{theorem}[Second Principle; Lusin]
\label{th:Lusin}
Let $E\subset \RR$ be a measurable set with finite measure and
let $f:E\to\ovr{\RR}$ be finite a.e. in $E$.
\par
Then, $f$ is measurable in $E$ if and only if, for every $\eps>0$, there exists a closed set $K \subseteq E$ such that $m(E \setminus K) < \eps$ and the restriction of $f$ to $K$ is continuous.
\end{theorem}

This is what we mean for \textit{nearly continuos}.
\vskip.1cm
The proof of Lusin's theorem is done by approximation by simple functions. A \textit{simple function}
is a measurable function that has a finite number of real values. 
If $c_1, \ldots, c_n$ are the \textit{distinct} values of a simple function $s$,
then $s$ can be conveniently represented as
$$
s = \sum\limits_{j = 1}^{n} c_{j}\mathcal{X}_{E_j},
$$
where $\mathcal{X}_{E_j}$ is the characteristic function of the set $E_j = \left\{ x \in E \textrm{ : } s(x) = c_j \right\}$. Notice that the $E_j$'s form a
covering of $E$ of pairwise disjoint measurable sets.
\par
Simple functions play 
a crucial role in Real Analysis; this is mainly due to the following result of which we shall
omit the proof.

\begin{theorem}[Approximation by Simple Functions]
\label{th:approximation}
Let $E\subseteq\RR$ be a measurable set and let $f: E\to [0,+\infty]$  be a measurable function.
\par
Then, there exists an increasing sequence of non-negative simple functions $s_n$ that 
converges pointwise to $f$ in $E$ for $n\to\infty$.
\par
Moreover, if $f$ is bounded, then $s_n$ converges to $f$ uniformly in $E$.
\end{theorem}

We can now give the proof of Lusin's theorem.

\begin{proof}
Any measurable function $f$ can be decomposed as $f=f^+-f^-$, where $f^+=\max(f,0)$ and 
$f^-=\max(-f,0)$ are measurable and non-negative functions. Thus, we can always suppose that $f$
is non-negative and hence, by Theorem \ref{th:approximation}, it 
can be approximated pointwise by a sequence of simple functions.
\par
We first prove that a simple function $s$ is nearly continuos. 
Since the sets $E_j$ defining $s$ are measurable, if we fix $\eps> 0$ we can find closed subsets $K_j$ of $E_j$ such that $m(E_j \setminus K_j) < \eps/n$ for $j=1,\dots, n$.
The union $K$ of the sets $K_j$ is also a closed set and, since the $E_j$'s cover $E$, we have that $m(E \setminus K) < \eps$. Since the closed sets $K_j$ are pairwise disjoint (as the $E_j$'s are pairwise disjoint) and $s$ is constant on $K_j$ for all $j=1,\dots,n$, we conclude that $s$ is continuous in $K$.
\par
Now, if $f$ is measurable and non-negative, 
let $s_n$ be a sequence of simple functions that converges pointwise to $f$ and fix an $\eps>0$. 
\par
As the $s_n$'s are nearly continuous, for any natural number $n$, there exists a closed set $K_n \subseteq E$ such that $m(E \setminus K_n) < \eps/{2^{n+1}}$ and $s_n$ is continuous in $K_n$. By Theorem \ref{th:Egoroff}, there exists a closed set $K_0 \subseteq E$ such that $m(E \setminus K_0) < \eps/2$ and $s_n$ converges uniformly to $f$ in $K_0$ as $n\to\infty$. Thus, in the set
$$
K = \bigcap\limits_{n=0}^\infty K_n
$$ 
the functions $s_n$ are all continuous and converge uniformly to $f$. Therefore $f$ is continuous in $K$
and
$$
m(E \setminus K) = m\Bigl(\bigcup\limits_{n=0}^\infty (E \setminus K_n)\Bigr) \leq \sum_{n=0}^{\infty} m(E \setminus K_n) < \eps.
$$
\par
Viceversa, if $f$ is nearly continuous, fix an $\eps>0$ and let $K$ be a closed subset of $E$ such that $m(E \setminus K) < \eps$ and $f$ is continuous in $K$. For any $t \in \RR$, we have:
 $$
L^*(f,t)= \left\{x \in K : f(x) \geq t\right\} \cup \left\{x \in E \setminus K : f(x) \geq t\right\}.
$$ 
The former set in this decomposition is closed, as the restriction of $f$ to $K$ is continuous, 
while the latter is clearly a subset of $E \setminus K$ and hence its outer measure must be less than $\eps$. By Theorem \ref{th:first}, $L^*(f,t)$ is measurable (for any $t\in\RR$), which means that $f$
is measurable. 
\end{proof}

\section{The fourth principle}

We shall now present alternative proofs of Theorems \ref{th:Egoroff} and \ref{th:Lusin}.  
They are based on a fourth principle, that corresponds to the following theorem.

\begin{theorem}[Fourth Principle]
\label{th:fourth}
Let $E\subset \RR$ be a measurable set with finite measure and let $f:E\to\ovr{\RR}$ be a measurable function.
\par
Then, $f$ is finite a.e. in $E$ if and only if, for every $\eps > 0$, there exists a closed set $K \subseteq E$ such that $m(E \setminus K) < \eps$ and $f$ is bounded on $K$.
\end{theorem}

This is what we mean for \textit{nearly bounded}.
\begin{proof}
If $f$ is finite a.e., we have that 
$$
m(\left\{x \in E : |f(x)| = \infty\right\})=0.
$$
As $f$ is measurable, $|f|$ is also measurable and so are the sets 
$$
L(|f|,n) = \left\{x \in E :|f(x)| > n\right\}, \ n\in\NN.
$$ 
Observe that the sequence of sets $L(|f|,n)$ is decreasing and
$$
\bigcap\limits_{n=1}^\infty L(|f|,n) = \left\{x \in E :|f(x)| = \infty\right\}.
$$ 
As $m(L(|f|,1)) \leq m(E) < \infty$, we can
apply the (downward) monotone convergence theorem and infer that
$$
\lim_{n \to \infty} m(L(|f|,n)) = m(\left\{x \in E :|f(x)| = \infty\right\}) = 0.
$$ 
\par
Thus, if we fix $\eps>0$,  there is an $n_{\eps} \in \mathbb N$ such that $m(L(|f|,n_{\eps})) < \frac{\eps}{2}$. Also, we can find a closed subset $K$ of the measurable set
$E \setminus L(|f|,n_{\eps})$ such that $m(E \setminus L(|f|,n_{\eps}))-m(K)< \frac{\eps}{2}$. Finally, since $K \subseteq E \setminus L(|f|,n_{\eps})$, $|f|$ is obviously bounded by $n_\eps$ on $K$
and
$$
m(E\setminus K)=m(E \setminus L(|f|,n_{\eps}))+m(L(|f|,n_{\eps})\setminus K)<\eps.
$$
\par
Viceversa, if $f$ is nearly bounded, then for any $n \in\NN$ there exists a closed set $K_n \subseteq E$ such that $m(E \setminus K_n) < 1/n$ and $f$ is bounded (and hence finite) in $K_n$.
Thus, $\left\{x \in E :|f(x)| =\infty\right\} \subseteq E \setminus K_n$ for any $n \in\NN$, and hence $$
m(\left\{x \in E:|f(x)| = \infty\right\}) \leq \lim_{n\to\infty} m(E \setminus K_n) =0,
$$ 
that is $f$ is finite a.e..
\end{proof}

\begin{remark}
Notice that this theorem can also be derived from Theorem \ref{th:Lusin}. 
In fact,
without loss of generality, the closed set $K$ provided by Theorem \ref{th:Lusin} can be 
taken to be compact and hence, $f$ is surely bounded on $K$, being continuous on a compact set.
\end{remark}

More importantly for our aims, Theorem \ref{th:fourth} enables us to prove Theorem \ref{th:Lusin} \textit{without} using Theorem \ref{th:Egoroff}.

\begin{proof}[Alternative proof of Lusin's theorem]
The proof runs similarly to that presented in Section 2. If $f$ is measurable, without loss of generality, we can assume that $f$ is non-negative and hence $f$ can be approximated pointwise by a sequence of simple functions $s_n$, which we know are nearly continuous. Thus, for any $\eps>0$, we can still construct the sequence of closed subsets $K_n$ of $E$ such that $m(E \setminus K_n) < \eps/{2^{n+1}}$ and $s_n$ is continuous in $K_n$. 
\par
Now, as $f$ is finite a.e., Theorem \ref{th:fourth} implies that it is nearly bounded, that is we can find a closed subset $K_0$ of $E$ in which $f$ is bounded and $m(E \setminus K_0) < \eps/ 2$. 
We apply the second part of the Theorem \ref{th:approximation} and infer that $s_n$ converges uniformly to $f$ in $K_0$. As seen before, we conclude that $f$ is continuous in the intersection $K$ of all the $K_n$'s,
because in $K$ it is the uniform limit of  the sequence of continuous functions $s_n$. As before
$m(E\setminus K)<\eps$.
\par
The reverse implication remains unchanged.
\end{proof}

In order to give our alternative proof of Theorem \ref{th:Egoroff}, we need to recall
a classical result for sequences of continuous functions.

\begin{theorem}[Dini]
\label{th:Dini}
Let $K$ be a compact subset of $\RR$ and let be given a sequence of continuous functions $f_n:K\to \RR$ that converges pointwise and monotonically in $K$ to a function $f:K\to\RR$.
\par
If $f$ is also continuous, then $f_n$ converges uniformly to $f$.
\end{theorem}

\begin{proof}
We shall prove the theorem when $f_n$ is monotonically increasing. 

For each $n \in \NN$, set $h_n = f - f_n$; as $n\to\infty$ the continuos functions $h_n$ decrease pointwise to $0$ on $K$.

Fix $\eps>0$. The sets $A_n = \left\{x \in K:h_n(x) < \eps\right\}$ are open, since the $h_n$'s are continuous; also,  $A_n \subseteq A_{n+1}$ for every $n\in\NN$, since the ${h_n}$'s decrease; finally,
the $A_n$'s cover $K$, since the $h_n$ converge pointwise to $0$. 
\par
By the compactness, $K$ is then covered by a finite number $m$ of the $A_n$'s, which means that $A_m = K$ for some $m\in\NN$. This implies that  $|f(x) - f_n(x)| < \eps$ for all $n \geq m$ and $x \in K$, as desired.
\end{proof}

\begin{remark}
\label{semicont}
The conclusion of Theorem \ref{th:Dini} still holds true if we assume that the sequence of $f_n$'s is increasing (respectively decreasing) and $f$ and all the $f_n$'s are lower (respectively upper) semicontinuous\footnote{We say that $f$ is lower (respectively upper) semicontinuous if the level sets

 $\{ x\in E: f(x)>t\}$ (respectively $\{ x\in E: f(x)<t\}$) are open for every $t \in \RR$.}.
\end{remark}

Now, Theorem \ref{th:Egoroff} can be proved by appealing for Theorems \ref{th:Lusin} and \ref{th:Dini}.

\begin{proof}[Alternative proof of Egoroff's theorem]
As in the classical proof of this theorem, we can always assume that $f_n(x) \to f(x)$ for every $x\in E$.
\par
Consider the functions and sets defined in \eqref{defgn} and \eqref{defEnm}, respectively.
We shall first show that there exists an $\nu\in\NN$ such that $g_n$ is nearly bounded for every
$n\ge \nu$. In fact, as already observed, since $g_n\to 0$ pointwise in $E$ as $n\to\infty$, we have that  
$$
E=\bigcup_{n=1}^\infty E_{n,1},
$$
and the $E_{n,1}$'s increase with $n$. Hence, if we fix $\eps>0$, there is a $\nu\in\NN$ such that $m(E\setminus E_\nu)<\eps/2$. Since
$E_\nu$ is measurable, by Theorem \ref{th:first} we can find a closed subset $K$ of $E_\nu$ such that
$m(E_\nu\setminus K)<\eps/2$.
\par
Therefore, $m(E\setminus K)<\eps$ and for every $n\ge \nu$
$$
0\le g_n(x)\le g_\nu(x)<1, \ \mbox{ for any } \ x\in K.
$$
\par
Now, being $g_n$ nearly bounded  in $E$ for every $n\ge \nu$, the alternative proof of Theorem \ref{th:Lusin} implies that $g_n$ is nearly continuous in $E$, that is for every $n\ge \nu$ there exists a closed
subset $K_n$ of $E$ such that $m(E\setminus K_n)<\eps/2^{n-\nu+1}$ and $g_n$ is continuous on $K_n$. 
The set 
$$
K=\bigcap_{n=\nu}^\infty K_n
$$
is closed, $m(E\setminus K)<\eps$ and on $K$ the functions $g_n$ are continuos for any $n\ge \nu$ 
and monotonically descrease to $0$ as $n\to\infty$.
\par
By Theorem \ref{th:Dini}, the $g_n$'s converge to $0$ uniformly on $K$.
This means that the $f_n$'s converge to $f$ uniformly on $K$ as $n\to\infty$. 
\par
The reverse implication remains unchanged.
\end{proof}

\begin{remark}
Egoroff's theorem can be considered, in a sense, as the \textit{natural} substitute of Dini's theorem, 
in case the monotonicity assumption is removed.
In fact, notice that the sequence of the $g_n$'s defined in \eqref{defgn} is decreasing; however, the $g_n$'s are in general no longer upper semicontinuous (they are only lower semicontinuous) and Dini's theorem (even in the form described in Remark \ref{semicont}) cannot be applied. In spite of that, the $g_n$'s \textit{remain} measurable if the $f_n$'s are so. 
\end{remark}

\section{Extensions.}
\label{sec:extensions}
Of course, all the proofs presented in Sections 2 and 3 work if we replace the real line $\RR$
by an Euclidean space of any dimension.
\vskip.1cm
Theorems \ref{th:Egoroff}, \ref{th:Lusin} and \ref{th:fourth}
can also be generalized replacing $\RR$ by a general measure space not necessarily endowed with a topology. 
\par
We recall that a measure space is a triple $(X,\cM,\mu)$. Here, $X$ is any set; $\cM$ is
a \textit{$\sigma$-algebra} in $X$, that is $\cM$ is a collection of subsets of $X$ that contains $X$ itself, the complement in $X$ of any set $E\in\cM$, and any countable union of sets $E_n\in\cM$
(the elements of $\cM$ are called {\it measurable sets});
$\mu$ is a function from $\cM$ to $[0,\infty]$ which is \textit{countably additive}, that is such that
\begin{equation*}
\label{additive}
\mu\Bigl(\bigcup_{n=1}^\infty E_n\Bigr)=\sum_{n=1}^\infty \mu(E_n),
\end{equation*}
for any sequence of pairwise disjoint sets $E_n\in\cM$.
\par
It descends from the definition that a measure $\mu$ is {\it monotone}, that is $\mu(E)\le\mu(F)$
if $E, F\in\cM$ and $E\subseteq F$. Another crucial property of a measure is the {\it monotone convergence theorem}: the sequence $\mu(E_n)$ converges to  
\begin{eqnarray*}
&&\mu\Bigl(\bigcup_{n=1}^\infty E_n\Bigr) \ \mbox{ if $E_n\subseteq E_{n+1}$ for any $n\in\NN$, or to} \\
&&\mu\Bigl(\bigcap_{n=1}^\infty E_n\Bigr) \ \mbox{ if $E_n\supseteq E_{n+1}$ for any $n\in\NN$
and $\mu(E_1)<\infty$.}
\end{eqnarray*}

In this general environment, Theorems \ref{th:Egoroff}, \ref{th:Lusin} and \ref{th:fourth}  can be extended simply by replacing closed sets by measurable sets; the proofs run similarly. 

\begin{theorem}
\label{thirdgeneralized}
Let ($X$,$\mathcal{M}$,$\mu$) be a measure space with $\mu(X)<\infty$.
\begin{itemize}
\item[(i)] {\rm (Egoroff-Severini)} A sequence of measurable functions $f_n:X\to\ovr{\RR}$  converges a.e. in $X$ to a measurable and finite a.e. function $f:X\to\ovr{\RR}$ if and only if, for every $\eps > 0$, there exists a measurable subset $E $ of $X$ such that $\mu(X \setminus E) < \eps$ and $f_n$ converges uniformly to $f$ on $E$.
\item[(ii)] {\rm (Lusin)} A finite a.e. function $f:X\to\ovr{\RR}$ is measurable in $X$ if and only if, for every $\eps > 0$, there exists a measurable subset $E$ of $X$ such that $\mu(X \setminus E) < \eps$ and the restriction of $f$ to $E$ is continuous.
\item[(iii)] {\rm (Fourth Principle)} A measurable function $f:X\to\ovr{\RR}$ is finite a.e. if and only if, for every $\eps > 0$, there exists a measurable subset $E$ of $X$ such that $\mu(X \setminus E) < \eps$ and $f$ is bounded on $E$.
\end{itemize}
\end{theorem}

Stated in this forms, the second, third and fourth principles elude the necessity of a First Principle that, of course, needs the presence of a topological space $(X,\tau)$ and the definition of a suitable outer measure on $X$. 
\par
We recall that on any set $X$ an outer measure $\mu_e$ can be defined as
a function on the power set $\mathcal{P}(X)$ with values in $[0,+\infty]$, which is monotone,
countably subadditive and such that $\mu(\varnothing)=0$. Carath\'eodory's extension theorem (see \cite{Ta}) then states that
one can always find a $\sigma$-algebra $\cM$ in $X$ (the $\sigma$-algebra of the so-called $\mu_e$-measurable sets) on which $\mu_e$ is actually a measure 
(that is $\mu_e$ is countably additive). Also, Carath\'eodory's criterion (see \cite{EG}) states that, if $\mu_e$ is a {\it Carath\'eodory measure}\footnote{That is, 
$\mu_e(E\cup F)=\mu_e(E)+\mu_e(F)$ for any choice of sets $E$ and $F$ such that $d(E,F)>0$.} on a metric space  $(X, d)$, then the $\sigma$-algebra of $\mu_e$-measurable
sets contains that of the Borel sets\footnote{That is the smallest $\sigma$-algebra that contains the topology in $(X, d)$.} 
and hence all the compact sets. 
\par
Whenever a First Principle is valid for a metric space $(X,d)$, the statements (classical and alternative) and proofs of Theorems \ref{th:Egoroff}, \ref{th:Lusin} and \ref{th:fourth} simply hold by replacing $\RR$ by $X$ and $m$ by $\mu_e$.

\vfill\eject

\end{document}